\documentclass[12pt,leqno,twoside]{amsart}

\usepackage[latin1]{inputenc}
\usepackage[T1]{fontenc}
\usepackage[colorlinks=true, pdfstartview=FitV, linkcolor=blue, citecolor=blue, urlcolor=blue]{hyperref}
\usepackage{amstext,amsmath,amscd, bezier,indentfirst,amsthm,amsgen,enumerate, geometry}
\usepackage[all,knot,arc,import,poly]{xy}
\usepackage{amsfonts,color, soul}  
\usepackage{amssymb}
\usepackage{latexsym}
\usepackage{epsfig}
\usepackage{graphicx}
\usepackage{srcltx}
\usepackage{enumitem}

\usepackage{tikz-cd}
\usepackage[all]{xy}

\usepackage{tikz} 
\usetikzlibrary{through} 
\usetikzlibrary{patterns} 
\usetikzlibrary{intersections} 
\usetikzlibrary{matrix} 
\usetikzlibrary{cd} 

\newtheorem{theorem}{Theorem}[section]
\newtheorem{corollary}[theorem]{Corollary}

\newtheorem{proposition}[theorem]{Proposition}
\theoremstyle{definition}
\newtheorem{definition}[theorem]{Definition}
\theoremstyle{remark}
\newtheorem{remark}[theorem]{\sc Remark}
\newtheorem{example}[theorem]{\sc Example}

\theoremstyle{claim}

\newcommand{\cod}{{\rm{cod\hspace{2pt}}}}

\newcommand{\Sing}{{\rm{Sing\hspace{2pt}}}}
\newcommand{\codSing}{{\rm{codSing\hspace{2pt}}}}

\begin{document}

\title[On flags of holomorphic foliations associated with singular second-order ode's]{On flags of holomorphic foliations associated with singular second-order ordinary differential equations}

\author[F. Lourenço]{Fernando Lourenço}
\address{Universidade Federal de Lavras - UFLA, Brazil}
\email{fernando.lourenco@ufla.br}

\author[F. Reis]{Fernando Reis}
\address{Universidade Federal do Espírito Santo - UFES, Brazil}
\email{fernando.reis@ufes.br}

\author[E. da Silva]{Euripedes da Silva}
\address{Instituto Federal do Ceará - IFCE, Brazil}
\email{euripedes.carvalho@ifce.edu.br}

\begin{abstract}
We consider germs of holomorphic vector fields at the origin of $\mathbb{C}^3$, with non-isolated singularities that are tangent to a holomorphic foliation of codimension one. This configuration is known as a $2$-flag of foliations. The focus is on cases where this geometric structure originates from second-order ordinary differential equations. We investigate the behavior of the singular sets associated with the foliations under consideration. Furthermore, we present a classification for second-order equations that admit a $2$-flag of foliations. Finally, we  propose a general method for constructing germs of $2$-flags of foliations at the origin of $\mathbb{C}^n$, with suitable properties of the singular sets, and we conclude by demonstrating that under generic assumptions, every equation of order greater than or equal to two is associated formally with a germ of 2-flag of holomorphic foliations at  $(\mathbb{C}^3,0)$.
\vspace{0.5 cm}

\noindent {\bf{Keywords:}} Holomorphic foliations, Flags, Vector fields, Non-isolated singularities, Second-order ordinary differential equations.
 
\end{abstract}

\maketitle

\section{Introduction}
 This paper presents an investigation of the theory of second-order ordinary differential equations of the form
  \begin{equation}\label{eq: second order1} u''(t)=\frac{P(u(t),u'(t),t)}{Q(u(t),u'(t),t)},
   \end{equation} 
where $P,Q:U\to \mathbb{C}$ are holomorphic functions defined on an open subset $U\subset \mathbb{C}^{3}$, i.e., $P,Q \in \mathcal{O}(U)$. This type of equation can be represented as a germ of a vector field at $ (\mathbb{C}^3,0)$, denoted by $X$, which admits a \textit{non-isolated singular set} and, it determines a germ of singular one-dimensional holomorphic foliation $\mathcal{F}_{1}$, such that its leaves projected to the graphs of the solutions of \eqref{eq: second order1} (see Section \ref{Section: ODEand2flags}).

Equation \eqref{eq: second order1} belongs to a large class of equations that plays a pivotal role in the theory of differential equations having a vast number of applications in physics, engineering, and sciences in general. It is also known as \textit{constrained} or \textit{implicit} systems (see, for instance, \cite{okaConst, bodadila, sotomayorImp} and the references cited therein). Hence,  the mathematical models considered are direct representatives of existing physical models.

Recently, equation \eqref{eq: second order1} (and the vector field $X$) was considered in \cite{artigo mayc} for the case where $P$ and $Q$ are polynomials. The author explores the one-dimensional foliation $\mathcal{F}_1$ associated with \eqref{eq: second order1} on a complex projective space and presents several significant results. For instance, that \eqref{eq: second order1} does not possess algebraic solutions when the bidegree is sufficiently large.
 
 In this work, we consider a special geometric behavior: the graphs of the solutions of \eqref{eq: second order1} lie in a family of surfaces. More precisely, we assume that the leaves of $\mathcal{F}_1$ are contained on submanifolds that are themselves leaves of a  holomorphic foliation of codimension-one,  $\mathcal{F}_2$.  This geometric structure is known as a \textit{flag of foliations}. It has been extensively studied in recent decades and can be applied to various contexts. We refer to \cite{MauResid, Cano2014,DA_Wil,Feigin,Ferr_Lour,DanubiaTh, flagsMol,Elvia,Reeb2}. We concentrate exclusively on flags of holomorphic foliations associated with equations of the form \eqref{eq: second order1}. Let us clarify what we mean by this association.

A \textit{$2$-flag of a second-order ordinary differential equation} (see Definition \ref{def: flag2ODE}) is a pair of germs of holomorphic foliations $\mathcal{F}=(\mathcal{F}_1,\mathcal{F}_2)$ such that there are an open subset $U\subset \mathbb{C}^3$, with $0\in U$ and a coordinate system $(x_1,x_2,x_3)$ in $U$, such that $\mathcal{F}_1$ can be defined by the vector field
\[
X=x_{2}Q(x_1,x_2,x_3)\frac{\partial}{\partial x_1 }+P(x_1,x_2,x_3)\frac{\partial}{\partial x_{2}}+ Q(x_1,x_2,x_3)\frac{\partial}{\partial x_3},
\]
and $\mathcal{F}_2$ can be defined by the integrable 1-form 
\[
\omega = \sum_{j=1}^{3}A_jdx_j,
\]
where $P,Q,A_j\in \mathcal{O}(U)$, such that $P(0)=Q(0)=A_j(0)=0$, with $P,Q\not \equiv 0$ and $P,Q$ without common factors (see Example \ref{ex: nonsing} and Remark \ref{remark: nonsing}).

One of the main tools in the theory of differential equations is the concept of \textit{first integrals}. In the context of the current work, a \textit{first integral} for a second-order ordinary differential equation is a nonconstant germ of a (holomorphic, meromorphic, formal, \dots) function $f:(\mathbb{C}^3,0)\to (\mathbb{C},0)$ which is constant on the leaves of the germ of one-dimensional holomorphic foliation $\mathcal{F}_1$. In other words, the graphs of the solutions are contained in the level sets of the first integral. The existence of first integrals is a classical problem, commonly referred to as the \textit{Integrability problem}. Readers who are not familiar with these notions and results are referred to the classical book \cite{Arnold 1}. 

It is worth noting that every second-order ordinary differential equation that has a first integral admits a $2$-flag of foliations (see Example \ref{remark: df(X)}). However, there are many classes of equations (or foliations) that do not have any kind of first integrals (for instance, see \cite{Poincare1}). More precisely, regarding the approach we are considering in the present work, we emphasize \cite{RebeloReis}, which presents germs of vector fields at $(\mathbb{C}^3,0)$ that admit a formal first integral but do not admit any holomorphic first integral. Furthermore, in \cite{PinheiroReis}, the authors explore an example of two topologically equivalent foliations, where only one of them admits two independent holomorphic first integrals.

In a certain sense, the existence of flags associated with differential equations can be considered a weaker version of the Integrability problem, since we find equations that have flags, regardless of the existence of first integrals. In this case,  the orbits of the vector field $X$ (and then, the graph of the solutions of \eqref{eq: second order1}) will be contained in the leaves of a holomorphic foliation of codimension-one, such that it is not necessarily the level sets of a function (i.e., a first integral). 

Throughout the text, we address this weak version of the Integrability problem for several particular second-order equations by exhibiting the $2$-flag of foliations associated with them (see Examples  \ref{example: linearHILL}, \ref{example: Emden-Fowler},  \ref{ex: log}, \ref{example: firstNONlinear}, \ref{example: corPropQuasilinear}, \ref{ex: sing totalderivative tang to regular fol}, \ref{ex: formISOLATEDsing},  \ref{ex: noinclusion}). Proposition \ref{Prop: Quasilinear ode} provides necessary conditions for a quasilinear ordinary differential equation to admit a flag. In Section \ref{Section: Obstructions}, we prove three fundamental obstructions for the existence of $2$-flags of second-order ordinary differential equations.

One way to understand flags of foliations is to analyze the structure of the singular sets involved and the relations among them. In this direction, 
R. Mol proved in \cite{flagsMol}, Theorem 1, that $\Sing(\mathcal{F}_2) $ is invariant by $\mathcal{F}_1$ for any $2$-flag of holomorphic foliations $\mathcal{F}=(\mathcal{F}_1,\mathcal{F}_2)$ on a complex manifold. Furthermore, in the case where $\Sing(\mathcal{F}_2)$  is an isolated set then $\Sing(\mathcal{F}_2)\subseteq \Sing(\mathcal{F}_1)$. The converse inclusion has been examined under different assumptions (see \cite{MauResid}  Corollary 2, p.1168). 

In general, the singular set of a flag can be very complicated. Our first main result establishes conditions where the above inclusion also holds for non-isolated singularities.

\begin{theorem}\label{Main result}
	Let $\mathcal{F} = (\mathcal{F}_1,\mathcal{F}_2)$ be a $2$-flag of a second-order ordinary differential equation. Then, $\Sing(\mathcal{F})=\Sing(\mathcal{F}_1)$. 
	More precisely,
		\begin{equation}\label{eq: inclusion11} \Sing(\mathcal{F}_2) \subseteq  \Sing(\mathcal{F}_1). \end{equation}
\end{theorem}	

In Example \ref{ex: noinclusion}, we exhibit a $2$-flag of holomorphic foliations, which is not a 2-flag of a second-order ordinary differential equation, and the inclusion \eqref{eq: inclusion11} does not hold. Therefore, the condition that ensures the inclusion of the singular sets depends on the presence of the differential equation.

We also obtain the following classification.

	\begin{theorem}	\label{theo: B}
		Consider the equation
\begin{equation}\label{eq: n-order12}
	u''(t) = \frac{P(u(t),u'(t),t)}{Q(u(t),u'(t),t)}
\end{equation}
where $P$ and $Q$ are holomorphic functions, without common factors, on an open subset $U\subset \mathbb{C}^3$ such that $P(0)=Q(0)=0$, $P,Q\not\equiv 0$. Suppose that \eqref{eq: n-order12} admits a $2$-flag of a second-order ordinary differential equation $(\mathcal{F}_1,\mathcal{F}_2)$, such that $\codSing(\mathcal{F}_2)\neq 2$. Then, \eqref{eq: n-order12} can be reduced to a first-order ordinary differential equation. More precisely, we have only one of the following possibilities:
	\\
	
	\begin{itemize}
		\item[i)]  there exists a function $F\in \mathcal{O}(U)$  with an isolated critical point, such that \eqref{eq: n-order12} can be reduced to the (possibly implicit) first-order ordinary differential equation given by
		\[
		F(u(t),u'(t),t)=0.
		\]

		\item[ii)]  there is an analytic coordinate system that transforms \eqref{eq: n-order12} into a first-order ordinary differential equation, which can be written as
		\[
		y'(x) =\frac{f(x,y(x),\alpha)}{g(x,y(x),\alpha)},	
		\]
		for each $\alpha$, where  $f,g\in \mathcal{O}(U)$.
	\end{itemize}
\end{theorem}

Theorem \ref{theo: B} can be reformulated in terms of foliations.

\begin{theorem}\label{theo: Fol}
Let $\mathcal{F} = (\mathcal{F}_1,\mathcal{F}_2)$ be a $2$-flag of a second-order ordinary differential equation where $\cod(\Sing(\mathcal{F}_2))\neq 2$. Then, either $\mathcal{F}_1$ is a deformation of a (not necessarily integrable) 1-form $\omega = f(x,y,\alpha)dx - g(x,y,\alpha)dy$ on $(\mathbb{C}^2,0)\times (\mathbb{C},0)$ or $\mathcal{F}_1$ has a holomorphic first integral with an isolated critical point, where both cases cannot occur simultaneously.
\end{theorem}

Additionally, it is worth mentioning that a construction used in the proof of Theorem \ref{Main result} can be extended to 2-flags of foliations in spaces of any dimension. This extension enables us to construct germs of $2$-flags of foliations from a generic codimension-one foliation such that the inclusion \eqref{eq: inclusion11} remains valid. This generalization is only possible due to the structure of the differential equation, which means that the one-dimensional foliation in the flag is associated with an ordinary differential equation of order greater than or equal to two.

\begin{corollary}\label{Corollary: 01}
	Let  $\mathcal{G}$ be a germ of  holomorphic foliation of codimension one  at $(\mathbb{C}^n,0)$. Then there is a germ of one-dimensional holomorphic foliation $\mathcal{F}$, with $\cod(\Sing(\mathcal{F}))=2$, where $ (\mathcal{F},\mathcal{G})$ is a germ of $2$-flag of holomorphic foliations at $(\mathbb{C}^n,0)$ such that $\Sing(\mathcal{G}) \subseteq \Sing(\mathcal{F}).$  
\end{corollary}

On the other hand, although the hypothesis of associating a second-order ordinary differential equation with flags of foliations in three-dimensional space may seem highly restrictive, our last result points towards the direction that the higher-dimensional cases (and therefore, equations of higher-order) can be formally reduced to the case we are considering.  More precisely, as a consequence of Theorems  1.6 and 1.9 in \cite{bodadila}, regarding the classification of germs of vector fields with non-isolated singularities, we conclude this work with the following:

\begin{corollary}\label{Corollary: 02}
Consider the equation
	\begin{equation}\label{eq: n-order}
		u^{(n-1)}(t) = \frac{P(u(t),u'(t),\dots, u^{(n-2)},t)}{Q(u(t),u'(t),\dots, u^{(n-2)},t)}
	\end{equation}
	such that $n\geq 3$, $P$ and $Q$ are holomorphic functions, without common factors, on an open subset $U\subset \mathbb{C}^n$ such that $P(0)=Q(0)=0$, $P,Q\not\equiv 0$, and $D_{n-1}P(0) = -D_n Q(0)$ where $D_j$ represents the derivative with respect to the $j$th variable. Then, there is a formal transformation which reduces \eqref{eq: n-order} to a system of ordinary differential equations associated with a germ of $2$-flag of holomorphic foliations, $(\mathcal{F}_1,\mathcal{F}_2) $ at $(\mathbb{C}^3,0)$, where $\mathcal{F}_1$ can be determined by the vector field
	\[
	Y =  -x_1\frac{\partial}{\partial x_1} + x_2\frac{\partial}{\partial x_2}+x_1x_2\frac{\partial}{\partial x_3},
	\]  
	and $\mathcal{F}_2$ by the integrable 1-form
	\[
	\Omega = -x_2(x_1x_2-x_3)dx_1+x_1x_3 dx_2-x_1x_2dx_3,
	\]
in an analytic coordinate system $(x_1,x_2,x_3)$ on a representative open set $V\subset \mathbb{C}^3$.	
\end{corollary}
It is worth noting that the vector field $Y$, as presented in Corollary \ref{Corollary: 02}, was obtained in \cite{bodadila}.

\section{Preliminaries}\label{Section: pre}
In order to make the text as self-contained as possible, this section presents some basic notions about holomorphic foliations. For more details on the general theory of foliation (for instance, for the real case), we can cite \cite{Camacho1,Godbillon,Ily_Yako, ScarduaBook2021} and the several references therein.

\subsection{Holomorphic foliations (nonsingular case)} Let $M$ be a complex manifold of dimension $n$. A holomorphic foliation of codimension $ q \leq n$ of $M$ is a decomposition $\mathcal{F}$ of $M$ consisting of pairwise disjoint immersed complex submanifolds (called leaves) of dimension $n-q$, distributed as described below: for each point $x\in M$ there is a unique submanifold $L_x$ of the decomposition that passes by $x$ ($L_x$ is called the leaf through $x$). Furthermore, there exists a holomorphic chart $(\varphi,U)$ of $M$, with $x\in U$, $\varphi:U\to \varphi(U) \subset \mathbb{C}^n$, such that $\varphi(U)=D^{n-q}\times D^{q}$, where $D^{n-q}, D^{q}$ are open polydiscs in $\mathbb{C}^{n-q}$ and $\mathbb{C}^q$, respectively. If $L$ is a leaf with $L\cap U \neq \emptyset$, then $L\cap U = \cup_{y\in \Gamma_{L,U}} \varphi^{-1}(D^{n-q}\times \{y\})$, where $\Gamma_{L,U}$ is a countable subset of $D^{q}$.

There are another (equivalent) definitions of foliations (see Chapter 1 and Proposition 3.1.4 in \cite{ScarduaBook2021} for a detailed discussion). 
\subsection{The integrability condition}
A distribution of $k$-planes on a manifold $M$ is a map $D$ which associates to each point $x \in M$  a vector subspace of dimension $k$ of $T_{x}M$. This means that for every $q \in M$ there exist $k$-holomorphic vector fields $X_{1}, \ldots , X_{k}$ on $M$, such that $\{ X_{1}(x), \ldots , X_{k}(x)\}$ is a base for $D(x)$. One says that a distribution of $k$-planes $D$ is involutive if, given two vector fields $X$ and $Y$ on $U$ such that, for each $ x \in U$, $X(x)$ and $Y(x) \in D(x),$ then $[X,Y](x) \in D(x).$ 
Frobenius Theorem (\cite{Camacho1,Godbillon,ScarduaBook2021}) assures that if the distribution $D$ is involutive then $D$ defines a foliation on $M$ of  dimension $k$.

Now, let $\omega_{1},\ldots, \omega_{q} \in \Omega^{1}(M)$ be linearly independent   holomorphic $1$-forms on $M$. Then we have a corresponding distribution $D$ of $(n-q)$-dimensional planes given by
\[
D(x) = \{v\in T_pM, \omega_j(x)\cdot v =0,j=1,\dots,q\}
\]
for each $x\in M$.
We say that the system  $\omega_{1},\ldots, \omega_{q}$ is integrable if it satisfies
\begin{equation*}
	d\omega_{j} \wedge \omega = 0 \ \ \forall j=1,\ldots, q \ \ on \ \ U.
\end{equation*}
It follows from Frobenius Theorem that the system of $1$-forms $\omega_1,\dots,\omega_q$ is integrable if and only if 
the distribution $D$ defines a foliation on $M$ of dimension $n-q$.

\subsection{Germs of singular holomorphic foliations and flags }\label{sec_foliation}
Let $M$ be a complex manifold. A singular holomorphic foliation can be defined as a pair $\mathcal{F}=(\mathcal{F}^{*},S)$, where $S \subsetneqq M$ is an analytic subset and $\mathcal{F}^{*}$ is  a nonsingular holomorphic foliation of $M\setminus S$. We refer to $S$ as the singular set of $\mathcal{F}$ and denote it as $S =\Sing(\mathcal{F})$. The leaves of $\mathcal{F}$ are the leaves of $\mathcal{F}^*$ on $M\setminus \Sing(\mathcal{F})$. 

 Throughout this work, we focus on holomorphic foliations of dimension one and codimension one. Therefore, it is convenient to express the framework within this context.

\begin{proposition}(\cite {AlcidesBrunoBook} Propositions 1.5 and 1.11)\label{pro_reduced}
Let $\mathcal{F}$ be a singular holomorphic foliation on a complex manifold $M$ of dimension one or codimension one. There exists a foliation $\mathcal{F}_1$ in $M$ with the following properties:
\begin{itemize}
	\item[a)]the irreducible components of $\Sing(\mathcal{F}_1)$ are of codimension $\geq 2$ , where $\Sing(\mathcal{F}_1) \subset \Sing(\mathcal{F}_2)$.
	\\
	\item[b)] $\mathcal{F}_1$ coincides with $\mathcal{F}$ in $M\setminus \Sing(\mathcal{F})$.
	\\
	\item[c)] $\mathcal{F}_1$ is maximal, that is, if $\mathcal{F}_2$ is the foliation in $M$ satisfying (a) and (b), then $\mathcal{F}_2 = \mathcal{F}_1$.
	\\
\end{itemize}
\end{proposition}
The reduced properties \textit{a),b)}, and \textit{c)} in Proposition \ref{pro_reduced} aims to avoid superfluous singularities, namely, components of codimension one in the singular set of the foliation. Therefore, from now on, we only consider \textit{reduced foliations} as $\mathcal{F}_1$ given above (see \cite{suwa} Proposition 1.7, for the cases of foliations of arbitrary dimension).

In particular, a germ of one-dimensional holomorphic foliation $\mathcal{F}_1$ at $(\mathbb{C}^{n},0)$ can be determined by a germ of holomorphic vector field $X$, such that $\Sing(\mathcal{F}_1) = \Sing(X) = \{z:X(z) =0\}$. In this case, each integral curves of $X$ is a leaf of $\mathcal{F}_1$. Furthermore, a germ of holomorphic foliation of codimension one $\mathcal{F}_2$ at $(\mathbb{C}^n,0)$ can be determined by a germ of holomorphic $1$-form $\omega$ such $\Sing(\mathcal{F}_2) = \Sing(\omega) = \{z; \omega(z) = 0 \}$ and 
\begin{equation}\label{Frobe}
d\omega \wedge \omega \equiv 0, 
\end{equation}
that is, $\omega$ satisfies the integrability condition.  Now, we are ready to define a germ of $2$-flags of foliations. We say $ \mathcal{F} = (\mathcal{F}_{1}, \mathcal{F}_{2})$ is a germ of $2$-flag of foliations at $(\mathbb{C}^{n},0)$ if the following condition is hold:
\begin{equation}\label{flag equation}
	\omega(X) \equiv 0.
\end{equation}
We define the singular set of the germ of the $2$-flag $\mathcal{F}$ as the union of the singular sets of the involved foliation
$$\Sing(\mathcal{F}) = \Sing(\mathcal{F}_{1}) \cup \Sing(\mathcal{F}_{2}).$$
If we take a coordinate system $(x_{1}, \dots, x_{n})$ on a representative open set $U\subset \mathbb{C}^3$, we can write
\begin{equation}\label{vector field}
	X = f_{1}(x_{1},x_{2}, \dots, x_{n})\dfrac{\partial}{\partial x_{1}} + f_{2}(x_{1},x_{2},\dots, x_{n})\dfrac{\partial}{\partial x_{2}} + \dots +  f_{n}(x_{1},x_{2},\dots, x_{3})\dfrac{\partial}{\partial x_{n}},
\end{equation}

\noindent where $f_{i}$ are holomorphic functions on $U$ for $i=1,2,\dots, n$, that is, $f_{i} \in \mathcal{O}(U).$ Analogously, 
\begin{equation}\label{form}
	\omega = A_{1}(x_{1},x_{2}, \dots, x_{n})dx_{1} + A_{2}(x_{1},x_{2}, \dots, x_{n})dx_{2} + \dots  +A_{n}(x_{1},x_{2}, \dots, x_{n})dx_{3},
\end{equation}

\noindent where $A_{i}\in \mathcal{O}(U)$. Therefore, the condition ($\ref{flag equation}$) can be written as follows

$$ \omega(X) = \sum_{i=1}^{n}A_{i}f_{i} = 0.$$

Flags of foliations (and distributions) can be considered in more general contexts and have been studied by several authors in recent years (see \cite{ MauResid,Cano2014,Feigin,Ferr_Lour, DanubiaTh, flagsMol,suwa}). However, there are still open problems in flag theory. One example is the Rationality Conjecture for Flags,  (see \cite{MauResid}, p.1176), which derives from the same problem for foliations (\cite{Baum_Bott}, p.287). Another open question is how to compute the residue of a flag in general, addressed in works such as \cite{MauResid, Ferr_Lour} (see also the survey \cite{L&R}). Moreover, there are many topics closely related to flags that naturally appear in the theory of foliation. For instance, Marco Brunella proposed a conjecture (see \cite{Cerveau}, p.443) stating that a two-dimensional holomorphic foliation $\mathcal{F}$ on $\mathbb{P}^{3}$ either admits an invariant algebraic surface or it composes a flag of holomorphic foliations.

\section{Flags and Ordinary differential equations}\label{Section: ODEand2flags}

From now on we consider second-order ordinary differential equation
\begin{equation}\label{generalequation2}
	u''(t)=\frac{P\left(u(t),u'(t),t\right)}{Q\left(u(t),u'(t),t\right)}
\end{equation} 
defined on a domain (open and connected subset) $T\subset\mathbb{C}$ and $P,Q$ are holomorphic coefficients depending on the solution $u(t)$, $u'(t)$ and $t$, where $Q\not\equiv 0$. Through a classical method of the theory of differential equations, we can introduce a new
	parameter representing the first derivative and any ordinary differential equation
	is converted into a  first-order system of differential equation, or a vector field, which we denote by $X$. It is well known that the integral curves of $X$, correspond, via natural projection, to the solutions of equation \eqref{generalequation2}. It is also known that $X$ defines a holomorphic one-dimensional foliation, $\mathcal{F}_{1}$, such that it leaves project to the graphs of the solutions of \eqref{generalequation2}.

	As mentioned earlier, in the introduction, we refer to \cite{artigo mayc} for a study of the vector field $X$ in the specific case where $P$ and $Q$ are polynomials. The work focuses on considering $\mathcal{F}_1$ on a complex projective space. 
	
	Let us give this process in the context of holomorphic germs.
Consider $U\subset \mathbb{C}^3$ be an open subset such that $0\in U$. Fixing a coordinate system $(x_1,x_2,x_3)$ at $U$, define a vector field $X$ given by 
\begin{equation}\label{totalderivative} X=x_{2}Q(x_1,x_2,x_3)\frac{\partial}{\partial x_1 }+P(x_1,x_2,x_3)\frac{\partial}{\partial x_{2}}+ Q(x_1,x_2,x_3)\frac{\partial}{\partial x_3},
\end{equation}
where $P,Q\in \mathcal{O}(U)$ and  $Q\not\equiv 0$. It is easy to see that for any solution $u(t)$ of \eqref{generalequation2} the holomorphic function $\phi:T\to\mathbb{C}^3$ given by $$\phi(t)=\left(u(t),u'(t),t \right)$$ is an integral curve of $X$. Conversely, let $\gamma:T\to\mathbb{C}^3$ be an integral curve of $X$ given by $\gamma(s)=\left(\gamma_1(s),\gamma_{2}(s),\gamma_3(s) \right)$. Then 
\begin{eqnarray*}&&\frac{d}{ds}\gamma(s)=\left(\gamma_1'(s),\gamma_2'(s),\gamma_3'(s) \right) ,
\end{eqnarray*}
implies 
\[\gamma_{1}''(s)Q\left(\gamma_1(s),\gamma_1'(s),s+c\right)=P\left(\gamma_1(s),\gamma_1'(s),s+c\right)\]
and taking $t:=s+c$ if necessary, we have that $\gamma_1(t)$ is a solution of \eqref{generalequation2}.
Note that the integral curves of the vector field $X$ project to the graphs of the solutions of the second-order differential equation. The discussion above motivates the following definitions. 

\begin{definition}\label{def: onedimfol}
We say that a germ of one-dimensional holomorphic foliation $\mathcal{F}_1$ at $ (\mathbb{C}^3,0)$ is \textit{associated with a second-order ordinary differential equation} if there are a representative open subset $U\subset \mathbb{C}^3$, with $0\in U$ and a coordinate system $(x_1,x_2,x_3)$ on $U$, where $\mathcal{F}_1$ can be determined by a vector field $X$ given by
\[
X=x_{2}Q(x_1,x_2,x_3)\frac{\partial}{\partial x_1 }+P(x_1,x_2,x_3)\frac{\partial}{\partial x_{2}}+ Q(x_1,x_2,x_3)\frac{\partial}{\partial x_3},
\]
where $P,Q \in \mathcal{O}(U)$. 
\end{definition}

\begin{remark}\label{remark: nonsing}
	
It is important to note that  we are dealing with \textit{singular local theory}. This means that we  assume $P(0)=Q(0)=0$, $P,Q\not\equiv 0$ and $P,Q$ without common factors, in Definition \ref{def: onedimfol}. Then, since $\mathcal{F}_1$ is reduced (see Proposition \ref{pro_reduced}) we have \[
\Sing(\mathcal{F}_1) = \{P=0\}\cap \{Q=0\}.
\]	

\end{remark}

	\begin{example}[\textit{Non-singular case}]\label{ex: nonsing}
	Consider the equation
	\begin{equation}\label{eq: nonsing}	
		u''(t) = P(u(t),u'(t),t),
	\end{equation}
	where $P\in \mathcal{O}(U)$ is a holomorphic function for some open subset $U\subset \mathbb{C}^3$. In this case, the one-dimensional holomorphic foliation associated $\mathcal{F}_1$ is determined by the vector field
	\[
	X= x_2\frac{\partial}{\partial x_1} + P(x_1,x_2,x_3)\frac{\partial}{\partial x_2} + 1 \frac{\partial}{\partial x_3}.
	\]	
	Then, $\Sing(\mathcal{F}_1)=\emptyset$. It follows from \textit{Flow Box Theorem}, see (\cite{{Ily_Yako}}, Theorem 1.14. p.9), that there are an open subset $V\subset \mathbb{C}^3$ and a coordinate system $(y_1,y_2,y_3)$ at $0\in V$ such that $X=\frac{\partial }{\partial y_3}$.
\end{example}

Now, we will make a slight adaptation to the definition of flags in the context of second-order equations.

\begin{definition}\label{def: flag2ODE}
	A germ of \textit{$2$-flag of holomorphic foliations associated with a second-order ordinary differential equation}  is a pair of germs of holomorphic foliations $\mathcal{F} = (\mathcal{F}_1,\mathcal{F}_2)$ at $(\mathbb{C}^3,0)$, such that $\mathcal{F}_1$ is associated with a second order ordinary differential equation, determined by the vector field $X$ and $\mathcal{F}_2$ is a germ of codimension one holomorphic foliation induced by an integrable holomorphic $1$-form $\omega$  satisfying 
	$$\omega(X) = 0.$$
	\noindent For the sake of simplicity, we call such a pair as a \textit{$2$-flag of a second-order ordinary differential equation.} 
	
\end{definition}


	\begin{example}\label{ex: nonsingflag}
Consider the one-dimensional holomorphic foliation $\mathcal{F}_1$, in Example \ref{ex: nonsing}. Let $\mathcal{F}_2$ be the germ of codimension one holomorphic foliation determined by the $1$-form $\omega = dy_1$ on $V$, and let $\mathcal{F}_3$ be the germ of codimension one holomorphic foliation determined by the $1$-form $\nu=dy_2$ on $V$. Note that  $(\mathcal{F}_1,\mathcal{F}_2)$ and $(\mathcal{F}_1,\mathcal{F}_3)$ are $2$-flags of second-order ordinary differential equation.  
\end{example}

\begin{definition}\label{def: firstIntegral}
	A \textit{first integral} for a second-order ordinary differential equation is a nonconstant germ of a function $f:(\mathbb{C}^3,0)\to (\mathbb{C},0)$ which is constant on the leaves of the one-dimensional holomorphic foliation associated (see Definition \ref{def: onedimfol}).
\end{definition}

For those unfamiliar with these concepts, we recommend referring to the esteemed book \cite{Arnold 1} as a valuable resource.

	\begin{example}\label{remark: df(X)}
	Consider a second-order ordinary differential equation as in \eqref{generalequation2}. Let $\mathcal{F}_1$ be the germ of one-dimensional holomorphic foliation, represented by
\[
X=x_{2}Q(x_1,x_2,x_3)\frac{\partial}{\partial x_1 }+P(x_1,x_2,x_3)\frac{\partial}{\partial x_{2}}+ Q(x_1,x_2,x_3)\frac{\partial}{\partial x_3},
\]
where $P,Q \in \mathcal{O}(U)$. Suppose that equation \eqref{generalequation2} has a (holomorphic) first integral $f$. Then $df(X)\equiv 0$, (where $d$ is the exterior deriviative). Since $\omega = df$ define a germ of codimension one holomorphic foliation $\mathcal{F}_2$, therefore, the pair  $\mathcal{F}=(\mathcal{F}_1,\mathcal{F}_2)$ is a $2$-flag of a second-order ordinary differential equation. 
\end{example}

	\begin{example}\label{ex: twofirstInt}
	In Example \ref{ex: nonsing}, equation \eqref{eq: nonsing} admits two functionally independent first integrals. In this case, we say that the equation is a \textit{completely integrable system}.	 
	\end{example}

The Integrability problem is a classical and relevant question related to the existence of first integrals (see \cite{Arnold 1}). A relevant fact mentioned in the introduction is that many systems, equations (and foliations) are not completely integrable. Moreover, they may not have any kind of first integrals. There are numerous results in this direction. For instance, the classical Poincaré's nonintegrability theorem \cite{Poincare1} establishes a condition of the \textit{resonance} as an obstruction to the existence of first integrals (not even formal ones).
As another relevant result, we can cite \cite{LibreVallsZhang}, where the authors demonstrate that (almost everywhere) the completely integrable systems are linear. 
Specifically, in the context we are exploring in this work, we can refer \cite{RebeloReis}, which presents germs of vector fields at $(\mathbb{C}^3,0)$, which admits a formal first integral but do not admit any holomorphic first integral.

From now on, we address a weaker version of the Integrability problem and solve it for several equations by exhibiting  the $2$-flag of foliations associated with them. Let us begin with the linear case.

\begin{example}\label{example: linearHILL}
As a very special and particular case, we refer to \cite{FBH,FBT}. 
The authors study the linear second-order ordinary differential equation of the form
\begin{equation}\label{eq: lin}
u''(t) =- \frac{b(t)u'(t) + c(t)u(t)}{a(t)},
\end{equation}
with holomoprhic coefficients $a,b,c$, and the vector field associated with \eqref{eq: lin}, on an open set $U\subset\mathbb{C}^3$ given by $$X=x_2a(x_3)\frac{\partial}{\partial x_1} - (b(x_3)x_2+c(x_3)x_1)\frac{\partial}{\partial x_2}+ a(x_3)\frac{\partial}{\partial x_3}$$ (which is not necessarily linear).  They show, among other results, that the  1-form $$\omega =-a(x_3)x_2dx_1+a(x_3)dx_2+(a(x_3)x_2^2+b(x_3)x_1x_2+c(x_3)x_1^2)dx_3$$ is integrable and satisfies $\omega(X)\equiv 0$. We can rewrite the above saying that \textit{every linear equation as in \eqref{eq: lin} admits a $2$-flag of a second-order ordinary differential equation}. Additionally, the authors present conditions for the existence of a meromorphic first integral and demonstrate, in accordance with classical theory, that $\omega$ is a rational pullback of a Riccati foliation.
\end{example}

Now, we will examine the existence of a germ of a $2$-flag of a second-order ordinary differential equation associated with another classical equation, albeit \textit{nonlinear}.

\begin{example}[\textit{Emden-Fowler equation}]\label{example: Emden-Fowler} 
	Consider the equation
	\begin{equation}\label{eq: Emden-Fowler}
		u''(t)=kt^m u(t)^n,
	\end{equation}
	defined on a domain $T\subset \mathbb{C}$, where $k\in \mathbb{C}$ and $m,n\in \mathbb{Z}\setminus \{0\}$. 
	For $n=1$, this equation is linear (see Example \ref{example: linearHILL}). Let $U\subset \mathbb{C}^2\times T$ be an open set with $0\in U$. We fix an analytic coordinate system $(x_1,x_2,x_3)$ on $U$. If $m\cdot n>0$, the one-dimensional holomorphic foliation associated with equation \eqref{eq: Emden-Fowler} is non-singular (see Example \ref{ex: nonsing}). Suppose that $m<0$ and $n>0$. Then, consider the holomorphic foliation $\mathcal{F}_1$ determined by the germ of vector field $X$ which can be represented by
	\begin{equation*}
		X= x_2 x_3^{-m} \frac{\partial}{\partial x_1} +kx_1^n \frac{\partial}{\partial x_2} + x_3^{-m} \frac{\partial}{\partial x_3}
	\end{equation*}
on $U$. Define the $1$-form on $U$, given by 
	\begin{equation}
		\begin{split}
			\omega=  &   -((n-1)kx_1^nx_3+(m+n+1)x_2 x_3^{-m})dx_1\\ 
			& +x_3^{-m}\left((n-1)x_2x_3+(m+2)x_1\right)dx_2 -(k(m+2)x_1^{n+1}-(m+n+1)x_2^2x_3^{-m})dx_3.
		\end{split}
	\end{equation}
	It follows from a straightforward computation that 
	$\omega \wedge d\omega =0$ and $\omega(X)=0$. Analogously, if $m>0$ and $n<0$, then we consider the vector field
	\begin{equation*}
		X= x_2 x_1^{-n} \frac{\partial}{\partial x_1} +kx_3^m \frac{\partial}{\partial x_2} + x_1^{-n} \frac{\partial}{\partial x_3},
	\end{equation*} 
which also satisfies $\omega(X)\equiv 0$. Therefore, the Emden-Fowler equation admits a germ of $2$-flag of a second-order ordinary differential equation.  
\end{example}

Next, we will explore examples of nonlinear equations that, unlike the Emden-Fowler equation, explicitly depend on $u$, $u'$, and $t$ while admitting a $2$-flag of a second-order ordinary differential equation.

\begin{example}\label{ex: log}
In \cite{Duarte2001}, the equation

	\begin{equation}\label{eq: log1}
	u''(t) = -\frac{u'(t)(u(t) + tu'(t)(2u(t)+1))}{tu(t)},
\end{equation}
was considered in the real domain. The authors present a first integral which is not a (single-valued) holomorphic function. However, we obtain a $2$-flag of foliation associated with  \eqref{eq: log1}. More than that, we generalize the above situation. Let $f:V  \to \mathbb{C}$ be any holomorphic function on an open set $V\subset \mathbb{C}^2$ and $g:\mathbb{C}^3\to \mathbb{C}^2$ be the map $g(x,y,z) = (x,yz)$, in the Euclidean coordinate system $(x,y,z)$. Consider the nonlinear equation
	\begin{equation}\label{eq: log2}
	u''(t) = -\frac{u'(t)f \circ g (u(t),u'(t),t)}{tu(t)},
	\end{equation}
defined on a domain $T\subset \mathbb{C}$. Then, $f\circ g$ is holomorphic in some open set  $U\subset \mathbb{C}^3$.
Let $\mathcal{F}_1$ be the germ of one-dimensional holomorphic foliation associated with \eqref{eq: log2} and represented on $U$ by the vector field 
\[
X = x_1 x_2  x_3 \frac{\partial}{\partial x_1} - x_2(f\circ g(x_1,x_2,x_3)) \frac{\partial }{\partial x_2} + x_1 x_3\frac{\partial }{\partial x_3}.
\] 
Define the $1$-form 
\[
\omega = (f\circ g(x_1,x_2,x_3) -x_1)dx_1 + x_1x_3dx_2 + x_1x_2 dx_3.
\]
A straightforward computation shows that  $\omega(X)\equiv 0$ and $\omega \wedge d\omega \equiv 0$. Denoting by $\mathcal{F}_2$ the germ of codimension one holomorphic foliation determined by $\omega$, we have that $(\mathcal{F}_1,\mathcal{F}_2)$ is a $2$-flag of foliation associated with \eqref{eq: log2}.
\end{example}

\begin{example}\label{example: firstNONlinear}
	Consider the nonlinear equation 
	\begin{equation}\label{eq: firstNONLINExample}
		a_2(u'(t))u''(t)+a_1(u(t))u'(t)+a_3(t)=0,
	\end{equation}
	in a domain $T\subset \mathbb{C}$ where $a_1:U_1\to \mathbb{C}$, $a_2:U_2\to \mathbb{C} $ and $a_3:T\to \mathbb{C}$ are holomorphic functions defined on open sets $U_1,U_2,T\subset \mathbb{C}$, such that $0\in U_1\times U_2\times T\subset \mathbb{D}(0)$, for some polydisc $\mathbb{D}(0)$ on $\mathbb{C}^3$. Denote $U:=U_1\times U_2 \times T$. The one-dimensional holomorphic foliation $\mathcal{F}_1$ is determined by the vector field
	\begin{equation*}\label{eq: totalderive  firstNONLINexample}
		X=x_2a_2(x_2)\frac{\partial}{\partial x_1}-(a_3(x_3)+x_2a_1(x_1))\frac{\partial}{\partial x_2}+ a_2(x_2)\frac{\partial}{\partial x_3}
	\end{equation*}
	on $U$ . Define a germ of $1$-form $\omega$ on $U$, given by
	\begin{equation*}
		\omega = a_1(x_1)dx_1+a_2(x_2)dx_2+a_3(x_3)dx_3.
	\end{equation*}	
	
	\noindent Then, $\omega(X)=0$. Furthermore, since $d\omega=0$, in particular, $\omega$ is integrable and thus determines a codimension one holomorphic foliation, denoted by $\mathcal{F}_2$. Then, the equation \eqref{eq: firstNONLINExample} admits the germ of $2$-flag of a second-order ordinary differential equation $\mathcal{F} = (\mathcal{F}_1,\mathcal{F}_2)$. Moreover, by Poincaré's Lemma, see (\cite{carmo}, Theorem 1, p.20), $\omega$ is exact. That is, there exists a holomorphic function $f$ defined on $U_1\times U_2\times T\subset D$ such that $\omega=df$. In particular, $f$ is a first integral for equation \eqref{eq: firstNONLINExample}. 
\end{example}

To present a slightly more general class of examples, we find conditions for a quasi-linear second-order ordinary differential equation to admit a germ of $2$-flag of foliation.

\begin{proposition}\label{Prop: Quasilinear ode}
	Consider a nonlinear second-order equation of the form
	\begin{equation}\label{eq: quasilinear}
		a_2(u(t),t) u''(t)+a_1(u(t),t)u'(t)+a_3(u(t),t) =0,
	\end{equation}
	defined on a domain $T\subset \mathbb{C}$ where $a_j:U \to\mathbb{C}$ are nonzero holomorphic functions in an open $U\subset \mathbb{C} \times T$ with $0\in U$ and $a_j(0)=0$ for all $j\in \{1,2,3\}$. Suppose that there is an analytic coordinate system $(x_1,x_2,x_3)$ in an open set $V\subseteq U$, such that
	\begin{equation}\label{eq: hyp2 propQuasilin}
		\frac{\partial \left(a_1/a_2\right)}{\partial x_3} = \frac{\partial (a_3/a_2)}{\partial x_1}.
	\end{equation}
Then \eqref{eq: quasilinear} admits a germ of $2$-flag of a second-order ordinary differential equation.
\end{proposition}
\begin{proof}
The one-dimensional foliation $\mathcal{F}_1$ associated with \eqref{eq: quasilinear} can be determined by the vector field
	$$X= x_2a_2(x_1,x_3)\frac{\partial}{\partial x_1}-\left(x_2a_1(x_1,x_3) +a_3(x_1,x_3)\right)\frac{\partial}{\partial x_2} +a_2(x_1,x_3)\frac{\partial}{\partial x_3},$$ 
   on $V$. Now, a holomorphic $1$-form $\omega$ given by
	$$
	\omega = a_1(x_1,x_3) dx_1 + a_2(x_1,x_3)dx_2 + a_3(x_1,x_3) dx_3,$$
on $V$. Then, $\omega(X)=0$. 
It remains to prove that $\omega$ is integrable on $V$. Since $\frac{\partial a_j}{\partial x_2} =0 , \forall j,$ we have
		\begin{eqnarray*}
		\omega \wedge d\omega &=& -\left(-a_3\frac{\partial a_2}{\partial x_1} +a_2\frac{\partial a_3}{\partial x_1} -a_2\frac{\partial a_1}{\partial x_3} + a_1\frac{\partial a_2}{\partial x_3}\right)dx_{1}\wedge dx_{2} \wedge dx_{3} \\
		&=& -a_2^2\left( \frac{\partial}{\partial x_1}\frac{a_3}{a_2} - \frac{\partial}{\partial x_3}\frac{a_1}{a_2} \right)dx_{1}\wedge dx_{2} \wedge dx_{3}. 
	\end{eqnarray*}
	It follows from \eqref{eq: hyp2 propQuasilin} that $\omega \wedge d\omega =0$. Now, consider the codimension one holomorphic foliation $\mathcal{F}_2$ determined by $\omega$. Therefore, $(\mathcal{F}_1,\mathcal{F}_2)$ is a $2$-flag of a second-order ordinary differential equation associated with \eqref{eq: quasilinear}. 

\end{proof}

\begin{example}\label{example: corPropQuasilinear} It follows from Propostion \ref{Prop: Quasilinear ode} that \begin{equation*}
		u''(t)=-\frac{u^4(t)+2t u(t)^3 u'(t)}{2u(t)^2}
	\end{equation*} admits a germ of $2$-flag of a second-order ordinary differential equation.
\end{example}

\section{Obstructions for the existence of $2$-flags}\label{Section: Obstructions}
As  before, consider the $2$-flag of a second-order ordinary differential equation (see Definition \ref{def: flag2ODE}) given by the pair of germs of holomorphic foliations $\mathcal{F}=(\mathcal{F}_1,\mathcal{F}_2)$ such that there are an open subset $U\subset \mathbb{C}^3$, with $0\in U$ and a coordinate system $(x_1,x_2,x_3)$ in $U$, where $\mathcal{F}_1$ can be defined by the vector field
\[
X=x_{2}Q(x_1,x_2,x_3)\frac{\partial}{\partial x_1 }+P(x_1,x_2,x_3)\frac{\partial}{\partial x_{2}}+ Q(x_1,x_2,x_3)\frac{\partial}{\partial x_3},
\]
and $\mathcal{F}_2$ can be defined by the integrable 1-form 
\[
\omega = \sum_{j=1}^{3}A_jdx_j,
\]
with $P,Q,A_j\in \mathcal{O}(U),P(0)=Q(0)=A_j(0)=0$ and $P,Q$ without common factors.

According to classical notions in the theory of differential equations, in the case where $\frac{\partial P}{\partial x_3} =\frac{\partial Q}{\partial x_3} \equiv 0$, we say that the second-order ordinary differential equation is \textit{autonomous}. Similarly, the equation is \textit{trivially reduced to a first-order ordinary differential equation} if $\frac{\partial P}{\partial x_1} =\frac{\partial Q}{\partial x_1}\equiv 0$. Then, we have the following classifications and obstructions. 

\vspace{5mm} 

\begin{proposition}\label{proposition: first obstruction}
	Consider  the equation  
	\begin{equation} \label{eq:propfirstobstru}
		u''(t)=\frac{P(u(t),u'(t),t)}{Q(u(t),u'(t),t)}, 
	\end{equation}
where $P$ and $Q$ are holomorphic functions, without common factors, on an open subset $U\subset \mathbb{C}^3$ such that $P(0)=Q(0)=0$, $P,Q\not\equiv 0$ and $\mathcal{F} = (\mathcal{F}_{1}, \mathcal{F}_{2})$ is the $2$-flag associated with \eqref{eq:propfirstobstru}. Then
	\\
	\begin{enumerate}
		\item[i)] 	$A_2$ cannot be identically null; 
		
		\vspace{0.5 cm}
		
		\item[ii)] $A_3$ is identically null if, and only if, equation \eqref{eq:propfirstobstru} is autonomous;
		
		\vspace{0.5 cm}
		
		\item[iii)]  $A_1$ is identically null if, and only if,  equation \eqref{eq:propfirstobstru} is trivially reduced to a first-order ordinary differential equation. 
		
	\end{enumerate}
	
\end{proposition}
\begin{proof}

Suppose that $A_{2} \equiv 0$. Then,
\begin{equation*}
	\omega(X) = x_{2}QA_{1} + PA_{2}+QA_{3} =0.
\end{equation*}
This implies that $A_{3} = -x_{2}A_{1}$ in $U \setminus V(Q)$. According to the Identity Theorem, this equality also holds in $U$. However, this implies a superfluous component in the singular set of $\mathcal{F}_{2}$, that is,

\begin{eqnarray*}
	\omega & = & A_1dx_1 - x_2A_1dx_3\\
	& = & A_1(dx_1-x_2dx_3).
\end{eqnarray*}
In particular, if $V(A_1) \neq \emptyset$, then $\mbox{Sing}(\omega)=V(A_1)$, which has  $\dim_{\mathbb{C}} V(A_1)=2$. Consequently, $\cod\Sing(\omega)= 1$. Since $\mathcal{F}_2$ is a reduced foliantion, by Proposition \ref{pro_reduced}, we have $\cod \Sing(\omega)\geq 2$. Thus, we have a contradiction. Otherwise, if $V(A_1) = \emptyset$, then $\omega = dx_1-x_2dx_3$ is not integrable. Therefore, $A_2\not\equiv 0$. 

\noindent Now, to prove item $(ii)$, we suppose that $A_3 \equiv 0$ on $U$. Then, $\omega(X)\equiv 0$ is equivalent to $ x_2QA_1 = -PA_2$ and $x_2Q\omega = 
A_2\left(-Pdx_1+x_2Qdx_2\right)$. Let us define $\eta := -Pdx_1 + x_2Qdx_2$. Note that both $\omega$ and $\eta$ determine $\mathcal{F}_2$ on $U$. It follows from the integrability condition that
\begin{equation*}
	\begin{array}{ll}
		0 & = \eta \wedge d\eta  \\ \\
		& =\Big(x_{2}P\frac{\partial Q}{\partial x_3} -x_{2}Q\frac{\partial P}{\partial x_3}\Big)dx_{1}\wedge dx_{2}\wedge dx_{3} \\ \\
		& =x_{2}Q^2\frac{\partial}{\partial x_3}\left(\frac{P}{Q}\right)dx_{1}\wedge dx_{2}\wedge dx_{3}.
	\end{array}
\end{equation*}
Denote $W:= \{(x_1,x_2,x_3)\in U: Q(x_1,x_2,x_3)=0 \ \mbox{and} \ x_{2}=0\}$. Since $W\subset U$ is a proper analytic subset of a connected complex manifold $U$, $W$ is nowhere dense, and $U\setminus W$ is arcwise connected. Thus $\dfrac{\partial}{\partial x_3}\left(\dfrac{P}{Q}\right)\equiv 0 $ on $U\setminus W$.  Therefore, outside $W$ the meromorphic function $\frac{P}{Q}$ does not depend on the variable $x_3$. Since $P$ and $Q$ have no common factors, they do not depend on $x_3$. Conversely, suppose that the second-order equation is autonomous. Then, there exist holomorphic functions $p$ and $q$ defined on an open subset  $V$ of $ \mathbb{C}^2$ containing the origin $0\in \mathbb{C}^2$, such that $V \hookrightarrow  U \subset \mathbb{C}^3$ is given by 
\[
X=x_2q(x_1,x_2)\frac{\partial}{\partial x_1} + p(x_1,x_2)\frac{\partial}{\partial x_2}+q(x_1,x_2)\frac{\partial}{\partial x_3}.
\]  
We define a germ of $1$-form $\omega = -p(x_1,x_2)dx_1+x_2q(x_1,x_2)dx_2$ on $U$. Note that $\omega$ satisfies $\omega(X)=0$ and $\omega \wedge d\omega = 0$. The proof of item $(iii)$ follows in a similar manner. 
\end{proof}
\vspace{5mm} 

The following propositions serve as valuable tools in the classification of germs of $2$-flags associated with second-order ordinary differential equations.

\vspace{5mm}

\begin{proposition}\label{thm 10.1}
	Let  $(\mathcal{F}_1,\mathcal{F}_2)$ and $(\mathcal{G}_1,\mathcal{F}_2)$ be $2$-flags of second order differential equations. Then, $\mathcal{F}_1$ and $\mathcal{G}_1$ are representatives of the same germ.
	
\end{proposition}

\begin{proof}
	Let $\mathcal{F}_1$ and $\mathcal{G}_1$ be germs of one-dimensional holomorphic foliation at $ (\mathbb{C}^3,0)$ associated with two second-order ordinary differential equations. Consider the vector fields $X_1$ and $X_2$ that determine, $\mathcal{F}_1$ and $\mathcal{G}_1$, respectively. Then we have
	\begin{equation}\label{totalderivative2} X_j=x_2Q_j(x_1,x_2,x_3)\frac{\partial}{\partial x_1}+P_j(x_1,x_2,x_3)\frac{\partial}{\partial x_2}+Q_j(x_1,x_2,x_3)\frac{\partial}{\partial x_3},
	\end{equation}
	for each $j\in\{1,2\}$.  It follows from the hypothesis that $X_1$ and $X_2$ are tangent to the same germ of codimension one holomorphic foliation $\mathcal{F}_2$. Let $\omega$  be the germ of an integrable $1$-form that defines $\mathcal{F}_2$. We can represent $\omega=\sum_{j=1}^{3}A_j dx_j$ for some holomorphic functions $A_j \in \mathcal{O}(U)$, which implies that
	$\omega(X_j)=0,\forall j$. This implies
	
	$$x_2A_1Q_1+A_2P_1+A_3Q_1=0 \ \ \mbox{and} \ \ x_2A_1Q_2+A_2P_2+A_3Q_2=0.$$
	It follows from Proposition \ref{proposition: first obstruction} that $A_2\not\equiv 0$. Thus, there exists  $p\in U$ such that $A_2(p)\neq 0$, and we have $Q_2(p) P_1(p)=P_2(p)Q_1(p)$. Since $\dim(V(A_2)) = 2$ it follows from the Identity Theorem that $Q_2(p) P_1(p)=P_2(p)Q_1(p)$  holds for every $p\in U$. Consequently, $X_1 \wedge X_2 \equiv 0$. Therefore $X_1$
	and $X_2$ are linearly dependent.
\end{proof}

The following result shares similarities with Proposition 10 in \cite{DanubiaTh}, employing different proof techniques and specifically tailored to the context of $2$-flags of second-order ordinary differential equations.

\vspace{0.5 cm}

\begin{proposition}\label{thm 10.2} 	Let $(\mathcal{F}_1, \mathcal{F}_2)$ be a $2$-flag of second-order differential equation. If another $2$-flag, $(\mathcal{F}_1, \mathcal{F}_3)$, exists, it going to the only one where $\mathcal{F}_3$ is functionally independent of $\mathcal{F}_2$.
\end{proposition}

\begin{proof}

Let $X = x_2Q\frac{\partial}{\partial x_1} + P\frac{\partial}{\partial x_2} + Q\frac{\partial}{\partial x_3}$ and $\omega = \sum_{j=1}^{3}A_j dx_j$ be the representatives of $\mathcal{F}_1$ and $\mathcal{F}_2$, respectively. First, note that
\begin{equation}\label{eq: quocientPQ222}
\frac{P}{Q}=- \Big(\frac{x_2A_1+A_3}{A_2}\Big).
\end{equation}
Then, suppose that there exists a germ $\mathcal{F}_3$ of codimension one holomorphic foliation at $(\mathbb{C}^3,0)$ which is functionally independent of  $\mathcal{F}_2$ such that $X$ is again tangent to $\mathcal{F}_3$.  Consider an integrable 1-form $\omega_{\mathcal{F}_3}$ representing $\mathcal{F}_3$ on $U$ where we may write $\omega_{\mathcal{F}_3}=B_1dx_1+B_2dx_2+B_3dx_3$, for some holomorphic functions $B_j \in \mathcal{O}(U)$. It follows from Proposition \ref{proposition: first obstruction} that $B_2\not\equiv 0$. Then, since $Q\not\equiv 0$ we can write 

\begin{equation}\label{eq: quocientPQ2}
\frac{P}{Q}=- \Big(\frac{x_2B_1+B_3}{B_2}\Big).
\end{equation}
Equations \eqref{eq: quocientPQ222} and \eqref{eq: quocientPQ2} imply $$QA_3=-(x_2A_1Q+A_2P),$$ $$QB_3=-(x_2B_1Q+B_2P),$$
$$(x_2A_1+A_3)B_2=(x_2B_1+B_3)A_2,$$
and $$A_2B_3-B_2A_3=x_2(A_1B_2-A_2B_1).$$  Substituting these equations in $\omega\wedge \omega_{\mathcal{F}_3}$, we obtain
\begin{equation*}
	\begin{split}
		Q\omega\wedge \omega_{\mathcal{F}_3} = & Q(A_1B_2-A_2B_1)dx_1\wedge dx_2 +(A_2QB_3-QA_3B_2)dx_2\wedge dx_3+Q(A_1B_3-A_3B_1)dx_1\wedge dx_3 \\
		= & Q(A_1B_2-A_2B_1)dx_1\wedge dx_2 +x_2Q(A_1B_2-A_2B_1)dx_2\wedge dx_3 \\
		& +  P(A_{2}B_{1} - A_{1}B_{2})dx_1\wedge dx_3\\
		= & (A_1B_2-A_2B_1)(Qdx_1\wedge dx_2 +x_2Qdx_2\wedge dx_3-Pdx_1\wedge dx_3) \\
	\end{split}
\end{equation*}

\noindent that we may write as a short form

\begin{equation}\label{eq: omega12}
		Q\omega\wedge \omega_{\mathcal{F}_3}= (A_1B_2-A_2B_1) \eta,
\end{equation}
where $\eta = Qdx_1\wedge dx_2 +x_2Qdx_2\wedge dx_3-Pdx_1\wedge dx_3$. Now consider another germ of codimension one holomorphic foliation $\mathcal{H}$ at $(\mathbb{C}^3,0)$. Let $\omega_\mathcal{H}$ be an integrable $1$-form representing $\mathcal{H}$ on $U$ such that $\omega_\mathcal{H} = \sum_{j=1}^{3}C_jdx_j$. Then, we have
\begin{equation}\label{eq: omega123}
	\begin{split}
		Q\omega\wedge \omega_{\mathcal{F}_3}\wedge \omega_\mathcal{H} &= (A_1B_2-A_2B_1) \eta \wedge \omega_\mathcal{H}.
	\end{split}
\end{equation}
We also have
\begin{equation}\label{eq: omega3eta}
	\begin{split}
		\omega_\mathcal{H}\wedge \eta &= C_3Qdx_3\wedge dx_1\wedge dx_2 + x_2QC_1 dx_1\wedge dx_2\wedge dx_3 -PC_2dx_2\wedge dx_1\wedge dx_3 \\
		&= (x_2QC_1 + PC_2 + QC_3)dx_1\wedge dx_2\wedge dx_3 \\
		&=\omega_\mathcal{H}(X) dx_1\wedge dx_2 \wedge dx_3.
	\end{split}
\end{equation}
If $X$ is tangent to $\mathcal{H}$, then it follows from \eqref{eq: omega12}, \eqref{eq: omega123} and \eqref{eq: omega3eta} that $\omega\wedge\omega_{\mathcal{F}_3}\wedge \omega_\mathcal{H}\equiv 0$. 
\end{proof}

\section{Proof of Theorem \ref{Main result}}

\begin{proof}
	Let us recall some notations. Since the foliation $\mathcal{F}_1$ is associated with a second-order ordinary differential equation, it can be determined by a germ of the vector field $X$ at $(\mathbb{C}^3,0)$ represented in an open set $U\subset \mathbb{C}^3$ by
\[
X= x_2Q(x_1,x_2,x_3)\frac{\partial}{\partial x_1} + P(x_1,x_2,x_3)\frac{\partial}{\partial x_2}+Q(x_1,x_2,x_3)\frac{\partial}{\partial x_3},
\]
where $P,Q\in \mathcal{O}(U)$. The foliation $\mathcal{F}_2$ is a germ of codimension one holomorphic foliation at $(\mathbb{C}^3,0)$ and which can be represented by a germ of integrable $1$-form $\omega$ represented in $U$ by $\omega = \sum_{j=1}^3A_jdx_j$.  Note that $\mathcal{F}_2$ can be regular, that is, $\Sing(\mathcal{F}_2)=\emptyset$. In this case, tautologically, we have the inclusion. Suppose that $\Sing(\mathcal{F}_2)$
	is an isolated set. Without loss of generality, we can assume that $\Sing(\mathcal{F}_2) = \{0\}$. It follows from \cite{flagsMol} Corollary 1, that  $\Sing(\mathcal{F}_2) = \{0\}\subset \Sing(\mathcal{F}_1)$. 
Finally, suppose that $\Sing(\mathcal{F}_2)$ is non-isolated. It follows from Proposition \ref{proposition: first obstruction} that $A_2\not\equiv 0$. Now, define the vector field $Y$ given by
	\begin{equation*}
		Y = x_2A_2(x_1,x_2,x_3)\frac{\partial }{\partial x_1} -\left(x_2A_1(x_1,x_2,x_3)+ A_3(x_1,x_2,x_3)\right)\frac{\partial }{\partial x_2} +A_2(x_1,x_2,x_3)\frac{\partial }{\partial x_3}. 
	\end{equation*}
	
	\noindent By a straightforward computation, it is possible to show that $\omega(Y)\equiv 0$. This means that $Y$ is also tangent to $\mathcal{F}_2$. It follows from Proposition \ref{thm 10.1} that $Y$ and $X$ determine the same germs of holomorphic foliation $\mathcal{F}_1$. However, the vector field $Y$ can have some superfluous components in its singular set. Indeed, if we suppose that $A_2$ and $x_2A_1+A_3$ does not have any common factors then \begin{equation*}
		\begin{split}
			\Sing(\mathcal{F}_1)= \text{    } &\ \Big\{(x_1,x_2,x_3)\in U: x_2A_1(x_1,x_2,x_3)=-A_3(x_1,x_2,x_3) \ \mbox{and} \ A_2(x_1,x_2,x_3)=0 \Big\},\\
		\end{split}
	\end{equation*}
	which implies $\Sing(\mathcal{F}_2)\subseteq \Sing(\mathcal{F}_1)$. Now, suppose that $A_2$ and $x_2A_1+A_3$ have common factors. Then, there are functions $h,b_1,b_2:U\subset \mathbb{C}^3\to \mathbb{C}$ such that $A_2=hb_2$ and $x_2A_1+A_3=hb_1$. Since $\mathcal{F}_1$ is reduced, it follows from Proposition \ref{pro_reduced} that $\mathcal{F}_1$ is determined by a vector field given by $Z=x_2b_2(x_1,x_2,x_3)\frac{\partial }{\partial x_1}-b_1(x_1,x_2,x_3)\frac{\partial }{\partial x_2}+b_2(x_1,x_2,x_3)\frac{\partial }{\partial x_3}$. In this case we also have $\Sing(\mathcal{F}_2)\subseteq \Sing(\mathcal{F}_1)$.
\end{proof}

\section{Proof of Theorem \ref{theo: B}}

\begin{proof}
Since $\codSing(\mathcal{F}_2)\neq 2$ and we are considering reduced foliations (see Proposition \ref{pro_reduced}) then we have only two cases: either $\Sing(\mathcal{F}_2)=\emptyset$, or $\Sing(\mathcal{F}_2)=\{0\}$. In the former, there exist an analytic coordinate system $(y_1,y_2,y_3)$, and a representative open subset $V\ni 0$ of $(\mathbb{C}^3,0)$, such that $\omega=hdy_3$ in $V$, for some holomorphic function $h$, which does not vanish on $V$. Furthermore, there are holomorphic functions $f_j$ such that we can represent $X=\sum_{j=1}^{3}f_j(y_1,y_2,y_3)\frac{\partial}{\partial y_j}$ on $V$. Note that  $\omega(X)=0$, also holds on $V$. Hence, $$0=\omega(X) = 0\cdot f_1 + 0\cdot f_2 + hf_3 \equiv 0 \Rightarrow h\cdot f_3 \equiv 0.$$ Since $h\neq 0$ on $V$ then, $f_3\equiv 0$, and  $$X=f_1(y_1,y_2,y_3)\frac{\partial}{\partial{y_1}}+f_2(y_1,y_2,y_3)\frac{\partial}{\partial y_2}.$$ Write $V:=V_1\times V_2 \times V_3 \subset \mathbb{C}^3$, where each $V_j$ is an open subset of $\mathbb{C}$ with $0\in V_j$. For each fixed $\alpha\in V_3$, consider the family of vector fields $\{X_\alpha\}_{\alpha\in V_3}$ on $V_1\times V_2 \times \{\alpha\}$, given by $$X_\alpha:=f_1(y_1,y_2,\alpha)\frac{\partial}{\partial_{y_1}}+f_2(y_1,y_2,\alpha)\frac{\partial}{\partial_{y_2}}.$$ Each leaf $L_\alpha$ of the foliation determined by $\omega$ in $V$ is given by $L_\alpha = \{(y_1,y_2,y_3 ) \in V : y_3=\alpha \}$. Hence, for each $\alpha\in V_3$, $X_\alpha$ defines a planar vector field on $L_\alpha $.  We claim that each connected component of $\Sing(X)$ cannot be entirely contained in any leaf of $\mathcal{F}_2$. Indeed, let $C\subset \Sing(X)$ be a connected component of the singular set of $\mathcal{F}_1$ such $C\subset L_{\alpha}$, for some $\alpha\in V_3$. Then, $f_1(C)=0=f_2(C)$ in $L_{\alpha}$. However, since $X_{\alpha}$ defines a planar vector field at $L_\alpha \approx \mathbb{C}^2$, this would imply that $X_\alpha$ is identically null on $L_\alpha$. This proves the claim. Hence, for each leaf $L_\alpha$, either $\Sing(X)\cap L_\alpha =\emptyset$ or $\Sing(X)\cap L_\alpha $ is a set with isolated points. If $L_\alpha$ is a leaf such that $\Sing(X)\cap L_\alpha =\emptyset$ then $X_\alpha$ defines a nonsingular planar vector field on $L_\alpha$. Now suppose that $\Sing(X)\cap L_\alpha$ is a set with isolated points. For each $p\in \Sing(X)\cap L_\alpha$ take an open set $W_p\subset L_\alpha$ such that $\Sing(X_\alpha) \cap W_p =\{p\}$. Therefore, $X_\alpha$ induces a planar vector field on $W_p$ with an isolated singular set. Define the distribution given by the 1-form $\omega = f_2(y_1,y_2,y_3) dy_1 - f_1(y_1,y_2,y_3)dy_2$ on $V$. For each $\alpha$ we have $$\omega_\alpha = f_2(y_1,y_2,\alpha) dy_1 - f_1(y_1,y_2,\alpha)dy_2$$ as a deformation on $(\mathbb{C}^2,0)\times (\mathbb{C},0)$. In particular, we have an associated first-order ordinary differential equation given by $$\frac{dy_2}{dy_1} = \frac{f_1(y_1,y_2,\alpha)}{f_2(y_1,y_2,\alpha)}.$$

	\noindent Now, suppose that $ \Sing(\mathcal{F}_2) = \{0\}$. It follows from Malgrange's Theorem \cite{MalgrangeI}, that there exist germs of holomorphic functions $G$ and $F$ at  $(\mathbb{C}^3,0)$ such that $G\neq 0$. Note that $\nabla F (p)= \left(\frac{\partial F}{\partial x_1}(p),\frac{\partial F}{\partial x_2}(p),\frac{\partial F}{\partial x_3}(p)\right)=0 $ if, and only if, $p=0$. In particular, since $\Sing(\mathcal{F}_2)=\{0\}$ then, $F$ has an isolated critical point at $0$.  
	Define a germ of vector field $Y$ at $ (\mathbb{C}^3,0)$ that can be represented by
	\[
	Y=\left(x_2\frac{\partial F}{\partial x_2}\right)\frac{\partial}{\partial x_1} - \left(x_2\frac{\partial F}{\partial x_1}+\frac{\partial F}{\partial x_3}\right)\frac{\partial}{\partial x_2} + \left(\frac{\partial F}{\partial x_j}\right)\frac{\partial}{\partial x_3}.
	\]
	Let us denote by $\mathcal{F}_Y$ the germ of  one-dimensional foliation determined by $Y$.  
	It is immediate to verify that $\omega(Y)\equiv 0$. Therefore $(\mathcal{F}_Y,\mathcal{F}_2)$ is a $2$-flag of a second-order ordinary differential equation. It follows from Propositions \ref{proposition: first obstruction} and \ref{thm 10.1} that $\mathcal{F}_1 = \mathcal{F}_Y$. In particular, if $u$ is a solution of the second-order ordinary differential equation then it satisfies $F(u,u',t)= 0$.

	\noindent To conclude the proof, consider the $2$-flag of a second-order ordinary differential equation $(\mathcal{F}_1,\mathcal{F}_2)$, where $\mathcal{F}_2$ is regular, that is, $\Sing(\mathcal{F}_2)=\emptyset$. Assume that there exists another $2$-flag $(\mathcal{F}_1,\mathcal{F}_3)$ such that $\Sing(\mathcal{F}_3)=\{0\}$. As above, there are a coordinate system $(y_1,y_2,y_3)$ and an open set $V\ni 0$, such that $\mathcal{F}_1$ can be represented by
	\[
	X = f_1\frac{\partial }{\partial y_1} + f_2 \frac{\partial }{\partial y_2}, 
	\] 
	and $\mathcal{F}_2$ can be given by $\omega=dy_3$. Also, as we did previously, there exist germs of functions $F,G : (\mathbb{C}^3,0)\to (\mathbb{C},0)$, with $G\neq 0 $ such that $\mathcal{F}_3$ can be represented by $\eta = GdF$. Now, define the vector field given by
	\[
	Y = \frac{\partial F}{\partial y_2}\frac{\partial }{\partial y_1} - \frac{\partial F}{\partial y_1}  \frac{\partial }{\partial y_2}.
	\]
	Since $dF(X)\equiv 0$, we have
	\[
	X\wedge Y = -\left( f_1\frac{\partial F}{\partial y_1} + f_2 \frac{\partial F}{\partial y_2}\right) \equiv 0.
	\]
	This implies that $X$ and $Y$ are linearly dependent and they then determine the same foliation $\mathcal{F}_1$. On the other hand, since $\Sing(dF)= \{0\}$, we have that $\frac{\partial F}{\partial y_1}(p )= \frac{\partial F}{\partial y_2}(p)=0$ if, and only if, $p=0$. In particular, $\Sing(Y) = \{0\}$. However,  $\Sing(\mathcal{F}_1)$ is a non-isolated set. Therefore, the two cases $(i)$ and $(ii)$ cannot occur simultaneously.
\end{proof}

\section{Examples} 

\begin{remark} Note that Examples \ref{example: linearHILL}, \ref{example: Emden-Fowler}, \ref{example: firstNONlinear}, and \ref{example: corPropQuasilinear} satisfy the inclusion \eqref{eq: inclusion11} in Theorem \ref{Main result}.
\end{remark}

\begin{example}\label{ex: formISOLATEDsing}
	Consider the equation 
	\begin{equation}\label{eq: ex2}
		u''(t)=-\frac{u'(t)^2+t}{u(t)}
	\end{equation}
	defined on a domain $T\subset \mathbb{C}$ containing the origin $0\in \mathbb{C}$. Let $U\subset \mathbb{C}^2\times T$ be an open set of $\mathbb{C}^3$ containing the origin $0\in \mathbb{C}^3$.  Equation \eqref{eq: ex2} is an example of the case $(i)$  in Theorem \eqref{theo: B}. In fact, fix an analytic coordinate system $(x_1,x_2,x_3)$ on $U$. The germ of one-dimensional holomorphic foliation  $\mathcal{F}_1$ associated with \eqref{eq: ex2} can be determined by the vector field given 
	\begin{equation*}
		X= x_1 x_2\frac{\partial}{\partial x_1} -( x_2^2 + x_3)\frac{\partial}{\partial x_2} + x_1\frac{\partial}{\partial x_3}.
	\end{equation*}
	Define the 1-form $\omega = x_2dx_1+x_1dx_2+x_3dx_3$. Then $\omega$ is integrable and satisfies $\omega(X)=0$ with $\Sing(\omega)=\{0\}$. Note that $\Sing(\omega) \subset \Sing(X)$. Furthermore, $F= x_1 x_2 + \frac{x_3^2}{2}$ determine a germ of first integral for $\omega$ and $X$.  In particular, equation \eqref{eq: ex2} can be reduced to an implicit first order ordinary equation. 
\end{example}

\begin{example}\label{ex: sing totalderivative tang to regular fol}
	Consider the equation 
	\begin{equation}\label{eq: nonsingFORM}
		u''(t)=-\frac{t+u'(t)}{u'(t)(t^2+2u(t))},
	\end{equation}
	defined on a domain $T\subset \mathbb{C}$ containing the origin $0\in \mathbb{C}$. Let $U\subset \mathbb{C}^2\times T$ be an open set of $\mathbb{C}^3$ containing the origin $0\in \mathbb{C}^3$. Fix an analytic coordinate system $(x_1,x_2,x_3)$ on $U$. The one-dimensional holomorphic foliation  $\mathcal{F}_1$ associated with \eqref{eq: nonsingFORM} can be  determined by the vector field 
	\begin{equation*}
		X= x_2^{2}(x_3^2+2x_1)\frac{\partial}{\partial x_1}- (x_2+x_3)\frac{\partial}{\partial x_2} +(x_2x_3^2+2x_1x_2)\frac{\partial}{\partial x_3}. 
	\end{equation*}
	Define the $1$-form $
		\omega = dx_1 + x_2(x_3^2+2x_1)dx_2+x_3dx_3.$ We have $\omega \wedge d\omega =0 $ and $\omega(X)=0$. Denote by $\mathcal{F}_2$ the codimension one holomorphic foliation determined by $\omega$. The pair $\mathcal{F} = (\mathcal{F}_1,\mathcal{F}_2)$ is a $2$-flag of a second-order ordinary differential equation (associated with \eqref{eq: nonsingFORM}). Therefore, this is the case $(ii)$  in Theorem \eqref{theo: B}. Moreover, we have $\Sing(\mathcal{F}_2)\subset \Sing(\mathcal{F}_1)$.
\end{example}

\begin{example}\label{ex: StrictInclusion}
	Consider a $2$-flag of a second-order ordinary differential equation, as given by Proposition \ref{Prop: Quasilinear ode}. In this case, we observe that $\Sing(\mathcal{F}_2)\varsubsetneq\Sing(\mathcal{F}_1)$, indicating a strict inclusion between the singular sets of the holomorphic foliations $\mathcal{F}_1$ and $\mathcal{F}_2$.
\end{example}

We conclude by emphasizing that the inclusion \eqref{eq: inclusion11} in Theorem \ref{Main result}  is a distinctive property  of $2$-flags of a second-order ordinary differential equation.  
\begin{example}\label{ex: noinclusion} 
	Let $\mathcal{G}_1$ be a germ of one-dimensional holomorphic foliation with a representative determined by the linear vector field 
	\[Y= x_1 \frac{\partial}{\partial x_1}+  x_2\frac{\partial}{\partial x_2}+  x_3\frac{\partial}{\partial x_3}\]
	on an open set $U\subset \mathbb{C}^3$, with an coordinate system $(x_1,x_2,x_3)$ such that $0\in U$. Furthermore, consider the germ of codimension one holomorphic foliation $\mathcal{G}_2$ at $ (\mathbb{C}^3,0)$ determined, on $U$, by the integrable $1$-form
	\[\omega=x_2 x_3 dx_1+x_1 x_3 dx_2 -2 x_1 x_2 dx_3.\]
	Note that $(\mathcal{G}_1, \mathcal{G}_2)$ is a $2$-flag of holomorphic foliation (although it is not a $2$-flag of a second-order ordinary differential equation). However, $\Sing(\mathcal{G}_2) \not\subset \Sing(\mathcal{G}_1) = {0}$.

	As an application of Corollary \ref{Corollary: 01}, we can construct an one-dimensional foliation $\mathcal{F}_1$ determined by
	\[
	X = x_1x_2x_3 \frac{\partial}{\partial x_1} -  (x_2^2x_3-2x_1x_2)\frac{\partial}{\partial x_2}+  x_1x_3\frac{\partial}{\partial x_3}
	\]
	such that $(\mathcal{F}_1,\mathcal{G}_2)$ is a $2$-flag of a second-order ordinary differential equation satisfying
	\[ \Sing(\mathcal{G}_2)\subseteq \Sing(\mathcal{F}_1).
	\]
\end{example}

\section{Proof of Corollary \ref{Corollary: 01}}
Let $U\subset \mathbb{C}^n$ be an open set and an analytic coordinate system $(x_1,\dots, x_n)$, where $\mathcal{G}$ is given by the integrable 1-form $\omega=\sum_{j=1}^{n} A_j dx_j$. Define the vector field given by
\[
X=\sum_{j=1}^{n-2}x_{j+1}A_{n-1}\frac{\partial}{\partial x_j }-\left(A_n+\sum_{j=1}^{n-2}x_{j+1}A_j\right)\frac{\partial}{\partial x_{n-1}}+ A_{n-1}\frac{\partial}{\partial x_n},
\]
on $U$. If $A_{n-1}$ and $-\left(A_n+\sum_{j=1}^{n-2}x_{j+1}A_j \right)$ have a common factor, then there are holomorphic functions $f(x_1,x_2,x_3), B_{n-1},B_{n}$ such that $A_{n-1} = fB_{n-1}$ and $\left(A_n+\sum_{j=1}^{n-2}x_{j+1}A_j \right) = fB_n$. Now, we can write 
$
X= f Y
$,
where 
$$Y = \sum_{j=1}^{n-2}x_{j+1}B_{n-1} \frac{\partial}{\partial x_j} + B_n\frac{\partial}{\partial x_{n-1}} + B_{n-1}\frac{\partial}{\partial x_n}.$$ 
Then, we can define the foliation $\mathcal{F}_U$, on  $U$, as the one-dimensional holomorphic foliation given by $X$ if there are no common factors or $Y$, otherwise. It is immediate that, $(\mathcal{F}_U,\mathcal{G})$ is a $2$-flag on $U$. To conclude, consider the germ of one-dimensional holomorphic foliation $\mathcal{F}$, with representative $\mathcal{F}_U$.

\section{Proof of Corollary \ref{Corollary: 02}}
	Let $\mathcal{F}_X$ be the germ of  one-dimensional holomorphic foliation at $(\mathbb{C}^n,0)$, $n\geq 3$, intrinsically associated with equation \eqref{eq: n-order}, that is, the foliation determined by the germ of the vector field $X$ represented by
	\[
	X =\sum_{j=1}^{n-2}z_{j+1}Q\frac{\partial}{\partial z_j } + P\frac{\partial}{\partial z_{n-1}}+ Q\frac{\partial}{\partial z_n},
	\]
in an analytic coordinate system $(z_1,\dots,z_n)$ on $U\subset \mathbb{C}^n$. Note that $\cod(\Sing(\mathcal{F}_X) ) = 2$. Let $DX(0)$ the linear part of $X$ at $0\in U$. From a straightforward computation we can obtain that $DX(0)$ has two nonzero eigenvalues given by $$\lambda_1 = \frac{b_n}{2} + \frac{a_{n-1}}{2} + \frac{\sqrt{a_{n-1}^2-2 a_{n-1}b_n +4a_nb_{n-1} +b_n^2 }}{2}$$ and $$\lambda_2 = \frac{b_n}{2} + \frac{a_{n-1}}{2} - \frac{\sqrt{a_{n-1}^2-2 a_{n-1}b_n +4a_nb_{n-1} +b_n^2 }}{2},$$ 
where $a_j=\frac{\partial P}{\partial z_j}(0)$ and $b_j = \frac{\partial Q}{\partial z_j}(0)$. By hypothesis, we have that $a_{n-1} = -b_{n}$. This implies $\frac{\lambda_1}{\lambda_2}=-1$. It follows from  \cite{bodadila} (see Theorems 1.6 and 1.9) that $X$ can be transformed by a formal change of coordinates in a germ of a vector field $Z$ at $(\mathbb{C}^n,0)$ determined by
\[
Z =  -z_1\frac{\partial}{\partial z_1} + z_2\frac{\partial}{\partial z_2}+z_1z_2\frac{\partial}{\partial z_3}.
\]  
Consider the projection $\pi:\mathbb{C}^n\to \mathbb{C}^3$ given by $\pi(z_1,z_2,z_3,\dots,z_n) = (z_1,z_2,z_3)$, and denote by $\mathcal{F}_1$ the one-dimensional holomorphic foliation determine by the projection of $Z$ by $\pi$, denote by $$Y= -x_1\frac{\partial}{\partial x_1} + x_2\frac{\partial}{\partial x_2}+x_1x_2\frac{\partial}{\partial x_3}$$ on $V:=\pi(U)\subset \mathbb{C}^3$. Now, define the $1$-form 
\[
\Omega =  -x_2(x_1x_2-x_3)dx_1+x_1x_3 dx_2-x_1x_2dx_3,
\]
on $V$. Then, $\Omega(Y) =  -x_2(x_1x_2-x_3)(-x_1)+x_1x_3 (x_2)-x_1x_2(x_1 x_2) \equiv 0$. Furthermore, $d\Omega=2x_1x_2dx_1\wedge dx_2 - 2x_2dx_1\wedge dx_3-2x_1dx_2\wedge dx_3$ and \[\Omega \wedge d\Omega = 2x_1x_2(x_1x_2-x_3)dx_1\wedge dx_2\wedge dx_3 - 2x_1x_2x_3dx_2\wedge dx_1\wedge dx_3 -2x_1^2x_2^2dx_3\wedge dx_1\wedge dx_2 \equiv 0 .\] 

This implies that $\Omega$ determines a germ of codimension one holomorphic foliation $\mathcal{F}_2$ at $(\mathbb{C}^3,0)$, such that $(\mathcal{F}_1,\mathcal{F}_2)$ is a germ of $2$-flag of holomorphic foliations at $(\mathbb{C}^3,0)$.

\vspace{0.5 cm}

\vspace{0.5 cm}

\subsection*{Acknowledgments}  The second author was  partially supported by the FAPEMIG [grant number 38155289/2021].

\end{document}